\documentclass[12pt]{amsart}

\usepackage{amsmath,amssymb,amsthm,url}

\usepackage[margin=1in]{geometry}

\newtheorem{theorem}{Theorem}[section]
\newtheorem{corollary}[theorem]{Corollary}
\newtheorem{lemma}[theorem]{Lemma}
\newtheorem{proposition}[theorem]{Proposition}
\theoremstyle{remark} \newtheorem*{remark}{Remark}
\theoremstyle{definition} \newtheorem{definition}{Definition}[section]

\renewcommand{\pmod}[1]{\left(\mathrm{mod}\text{ }#1\right)}
\newcommand{\legendre}[2]{{\left ( \frac{#1}{#2} \right )}}

\usepackage{hyperref}
\usepackage{color}
\usepackage[notref,notcite]{showkeys}
\definecolor{labelkey}{RGB}{127,0,0}
\definecolor{refkey}{RGB}{127,0,0}

\usepackage{graphicx}
\usepackage{enumerate}
\newcommand{\Q}{{\mathbb Q}}
\newcommand{\hatphi}{{ \hat \phi }}

\newcommand{\Sel}{{\mathrm{Sel}}}

\newcommand{\Ftwo}{{\mathbb{F}_2}}
\newcommand{\T}{{\mathcal T}}

\newcommand{\ord}{{ \mathrm{ord} }}

\usepackage[all]{xy} 

\usepackage[OT2,T1]{fontenc}
\DeclareSymbolFont{cyrletters}{OT2}{wncyr}{m}{n}
\DeclareMathSymbol{\Sha}{\mathalpha}{cyrletters}{"58}

\numberwithin{equation}{section}

\title[The distribution of the Tamagawa ratio]{The distribution of the Tamagawa ratio in the family of elliptic curves with a two-torsion point}
\author{Zev Klagsbrun}
\address{Center for Communications Research, 4320 Westerra Court, San Diego, CA 92121}
\email{zdklags@ccrwest.org}

\author{Robert J. Lemke Oliver}
\address{Department of Mathematics, Stanford University, Building 380, Stanford, CA 94305}
\email{rjlo@stanford.edu}

\thanks{The second author is supported by an NSF Mathematical Sciences Postdoctoral Fellowship.}

\begin{document}

\begin{abstract}
In recent work, Bhargava and Shankar have shown that the average size of the $2$-Selmer group of an elliptic curve over $\mathbb{Q}$ is exactly $3$, and Bhargava and Ho have shown that the average size of the $2$-Selmer group in the family of elliptic curves with a marked point is exactly $6$.  In contrast to these results, we show that the average size of the $2$-Selmer group in the family of elliptic curves with a two-torsion point is unbounded.  In particular, the existence of a two-torsion point implies the existence of rational isogeny.  A fundamental quantity attached to a pair of isogenous curves is the Tamagawa ratio, which measures the relative sizes of the Selmer groups associated to the isogeny and its dual.  Building on previous work in which we considered the Tamagawa ratio in quadratic twist families, we show that, in the family of all elliptic curves with a two-torsion point, the Tamagawa ratio is essentially governed by a normal distribution with mean zero and growing variance.
\end{abstract}

\maketitle

\section{Introduction and statement of results}

In recent work \cite{BS2010}, Bhargava and Shankar showed that when all elliptic curves over $\mathbb{Q}$ are ordered by height, the average size of the 2-Selmer group is equal to 3.  Similar work by Bhargava and Ho \cite{BH2012} shows that the average size is six when the average is taken over all elliptic curves with a marked point. This result has the same flavor as that of Bhargava and Shankar, in that, after discounting for the known contribution of the marked point, the average size is three. Here, we consider the related case where the marked point is of order two. Unlike the case of the generic marked point (which is almost always of infinite order) considered by Bhargava and Ho,  the existence of this point affects the average size of the 2-Selmer group in an essential way - in particular, the average size is no longer bounded.

Given an elliptic curve $E/\mathbb{Q}$ with a rational isogeny $\phi\colon E\to E^\prime$ of degree $p$, one can associate to $E$ a finite $p$-group called the $\phi$-Selmer group, which we denote by $\mathrm{Sel}_\phi(E/\mathbb{Q})$ (see Section \ref{sec:selmer} for the definition).  Similarly, one can also associate to the dual isogeny $\hat\phi\colon E^\prime\to E$ the $p$-group $\mathrm{Sel}_{\hat{\phi}}(E^\prime/\mathbb{Q})$.  The {\bf Tamagawa ratio} is defined to be
\[
\mathcal{T}(E/E^\prime) := \frac{|\mathrm{Sel}_\phi(E/K)|}{|\mathrm{Sel}_{\hat\phi}(E^\prime/K)|}.
\]
In this work, we consider the distribution of $\mathcal{T}(E/E^\prime)$ as $E$ ranges over the set of elliptic curves with a rational two-torsion point.

Let $E_{A,B}:y^2=x^3+Ax^2+Bx$ denote a generic such curve, and let $\phi\colon E_{A,B}\to E^\prime_{A,B}$ be the degree two isogeny corresponding to the rational subgroup generated by the point $(0,0)$.  We are interested in the distribution of the (logarithmic) Tamagawa ratio
\[
t(A,B):=\mathrm{ord}_2 \mathcal{T}(E_{A,B}/E^\prime_{A,B}).
\]
Let $\mathcal{E}(X):=\{(A,B)\in\mathbb{Z}^2:|A|,B^2\leq X, A^2-4B\neq 0, \text{ and, if } p^4 \mid B, \text{ then }p^2\nmid A\}$ be the set of $A$ and $B$ in a box for which the model $E_{A,B}$ is minimal.  Our main theorem is that, as we vary over elements of $\mathcal{E}(X)$, $t(A,B)$ becomes normally distributed.

\begin{theorem}
\label{thm:tamagawa}
As $X\to\infty$, the set $\{t(A,B) : (A,B)\in\mathcal{E}(X)\}$ becomes normally distributed with mean $0$ and variance $2\log\log X$.  That is, for any $z\in\mathbb{R}$, we have that
\[
\lim_{X\to\infty} \frac{1}{\#\mathcal{E}(X)} \#\{(A,B)\in\mathcal{E}(X): t(A,B) \leq z \sqrt{2\log\log X} \} = \frac{1}{\sqrt{2\pi}} \int_{-\infty}^z e^{-t^2/2} dt.
\]
\end{theorem}

\begin{remark}
Lemma \ref{lem:density} below shows that $\#\mathcal{E}(X) \sim 4X^{3/2}/\zeta(6)$.
\end{remark}

This theorem has a nice consequence for the distribution of $2$-Selmer ranks of the elliptic curves $E_{A,B}$, owing to the fact that $|\mathrm{Sel}_\phi(E_{A,B}/\mathbb{Q})|$ is essentially a lower bound for $|\mathrm{Sel}_2(E_{A,B}/\mathbb{Q})|$.  As remarked above, for the family of all elliptic curves over $\mathbb{Q}$, Bhargava and Shankar \cite{BS2010} have shown that average size of the 2-Selmer group is exactly 3, and for the family of curves with a marked point, but where that point is not required to be torsion, Bhargava and Ho \cite{BH2012} have shown that the average size of the 2-Selmer group is exactly 6.  In contrast to these results, Theorem \ref{thm:tamagawa} implies the following corollary.

\begin{corollary}
\label{cor:2-selmer}
For any integer $r\geq 0$, we have that
\[
\liminf_{X\to\infty} \frac{1}{\#\mathcal{E}(X)} \#\{ (A,B)\in\mathcal{E}(X) : \mathrm{dim}_{\mathbb{F}_2}(\mathrm{Sel}_2(E_{A,B}/\mathbb{Q})\geq r \} \geq \frac{1}{2}.
\] 
In particular, the average size of $\mathrm{Sel}_2(E_{A,B}/\mathbb{Q})$ is unbounded.
\end{corollary}

\begin{remark}
Of course, Corollary \ref{cor:2-selmer} contradicts neither Bhargava and Shankar's result nor Bhargava and Ho's, as the set of elliptic curves with a two-torsion point is of density zero in either family. 
\end{remark}

\begin{remark}
In forthcoming work, Kane and the first author, using different techniques, show that the average size of $\mathrm{Sel}_\phi(E_{A,B}/\mathbb{Q})$ for $E_{A,B}\in\mathcal{E}(X)$ is $\asymp \sqrt{\log X}$, from which it follows that the average size of $\mathrm{Sel}_2(E_{A,B}/\mathbb{Q})$ is $\gg \sqrt{\log X}$.  
\end{remark}

In recent work \cite{KLO2013}, the authors considered the analogous problem in the family of quadratic twists and proved the analogue of Theorem \ref{thm:tamagawa}.  The key insight in that case is that the Tamagawa ratio is essentially an additive function, which could be studied by proving a variant of the classical Erd\H{o}s-Kac theorem.  For the family under consideration in this paper, the Tamagawa ratio is no longer an additive function.  However, it can be decomposed into two pieces which are individually additive.  We adapt the proof of the Erd\H{o}s-Kac theorem due to Billingsley \cite{Billingsley1974} to show that these two pieces are independently and normally distributed, from which Theorem \ref{thm:tamagawa} follows.  In forthcoming work \cite{KLO2014}, we consider in greater generality these joint Erd\H{o}s-Kac style theorems and we apply them to the study of simultaneous twists of elliptic curves.


\section{Selmer groups}
\label{sec:selmer}

We begin by briefly recalling the definition of the $\phi$-Selmer group of $E$. 

If $E(\Q)$ has a point $P$ of order two, then there is a two-isogeny $\phi:E \rightarrow E^\prime$ between $E$ and $E^\prime$ with kernel $C = \langle P \rangle$. We have a short exact sequence of $G_\Q$ modules \begin{equation} 0 \rightarrow C \rightarrow E(\overline{\Q}) \xrightarrow{\phi}  E^\prime(\overline{\Q}) \rightarrow 0\end{equation} which gives rise to a long exact sequence of cohomology groups \begin{equation*}0 \rightarrow C \rightarrow E(\Q) \xrightarrow{\phi} E^\prime(\Q) \xrightarrow{\delta} H^1(\Q, C) \rightarrow H^1(\Q, E) \rightarrow H^1(\Q, E^\prime) \ldots\end{equation*} The map $\delta$ is given by $\delta(Q)(\sigma)= \sigma(R) - R$ where $R$ is any point on $E(\overline{\Q})$ with $\phi(R) = Q$. 

This sequence remains exact when we replace $\Q$ by its completion $\Q_v$ at any place $v$, which gives rise to the following commutative diagram.

\begin{center}\leavevmode
\begin{xy} \xymatrix{
E^\prime(\Q)/\phi(E(\Q))  \ar[d] \ar[r]^\delta & H^1(\Q, C) \ar[d]^{Res_v} \\
E^\prime(\Q_v)/\phi(E(\Q_v))   \ar[r]^\delta_v & H^1(\Q_v, C) }
\end{xy}\end{center}

We define a distinguished local subgroup $H^1_f(\Q_v, C) \subset H^1(\Q_v, C)$ as the image $$\delta_v \left ( E^\prime(\Q_v)/\phi(E(\Q_v)) \right ) \subset H^1(\Q_v, C)$$ for each place $v$ of $\Q$ and we define the $\mathbf{\phi}$\textbf{-Selmer group} of $E/\Q$, denoted $\Sel_\phi(E/\Q)$, by \begin{equation*}\Sel_\phi(E/\Q) = \ker \left ( H^1(\Q, C) \xrightarrow{\sum res_v} \bigoplus_{v\text{ of } \Q} H^1(\Q_v, E[2])/H^1_f(\Q_v, C) \right ).\end{equation*}
%





The isogeny $\phi$ on $E$ gives gives rise to a dual isogeny $\hat \phi$ on $E^\prime$ with kernel $C^\prime = \phi(E[2])$. Exchanging the roles of $(E, C, \phi)$ and $(E^\prime, C^\prime, \hat \phi)$ in the above defines the $\mathbf{\hat \phi}$\textbf{-Selmer group}, $\Sel_\hatphi(E^\prime/\Q)$, as a subgroup of $H^1(\Q, C^\prime)$. The groups $\Sel_\phi(E/\Q)$ and $\Sel_\hatphi(E^\prime/\Q)$ are finite dimensional $\Ftwo$-vector spaces and their ranks are related to that of the $2$-Selmer group $\Sel_2(E/\Q)$ via the following theorem.



\begin{theorem}\label{gss}The $\phi$-Selmer group, the $\hat \phi$-Selmer group, and the 2-Selmer group sit inside the exact sequence \begin{equation}0 \rightarrow E^\prime(\Q)[2]/\phi(E(\Q)[2]) \rightarrow \Sel_\phi(E/\Q) \rightarrow \Sel_2(E/\Q) \xrightarrow{\phi}\Sel_\hatphi(E^\prime/\Q).\end{equation}
\end{theorem}
\begin{proof}
This is a well known diagram chase. See Lemma 2 in \cite{FG2008} for example.
\end{proof}

\section{Tamagawa Ratios}
\label{sec:tamagawa ratios}

Our methods take advantage of a natural duality which exists between the groups $\Sel_\phi(E/\Q)$ and $\Sel_{\hat\phi}(E/\Q)$. This global duality is a consequence of a local duality between the distinguished local conditions $H^1_\phi(\Q,C)$ and $H^1_\hatphi(\Q,C^\prime)$ which is established in the following two lemmas.

\begin{lemma}\label{Re4.7}
The sequence \begin{equation}\label{ss} 0 \rightarrow C^\prime/\phi \left ( E(\Q_v)[2]  \right ) \xrightarrow{\delta_v} H^1_\phi(\Q_v, C) \rightarrow H^1_f(\Q_v, E[2]) \xrightarrow{\phi} H^1_{\hat \phi}(\Q_v, C^\prime) \rightarrow 0\end{equation} is exact.
\end{lemma}
\begin{proof}
This is a well-known result. See Remark X.4.7 in \cite{Silverman2009} for example.
\end{proof}

\begin{lemma}[Local Duality]\label{localduality}
For each place $v$ of $\Q$ there is a local Tate pairing $H^1(\Q_v, C) \times H^1(\Q_v, C^\prime) \rightarrow \{\pm 1 \}$ induced by a pairing $[ \text{   }, \text{   } ] :C \times C^\prime  \rightarrow \{\pm 1 \}$ given by $[ Q , \tilde R ]  =  \langle Q, R \rangle$, where $\langle Q, R \rangle$ is the Weil pairing and $R$ is any pre-image of $\tilde R$ under $\phi$. The subgroups defining the local conditions $H^1_\phi(\Q_v, C)$ and $H^1_{\hat \phi}(\Q_v, C^\prime)$ are orthogonal complements under this pairing.
\end{lemma}
\begin{proof}
Orthogonality  is equation (7.15) and the immediately preceding comment in \cite{Cassels1965}. Counting dimensions of the terms in (\ref{ss}) shows that $H^1_\phi(\Q_v, C)$ and $H^1_{\hat \phi}(\Q_v, C^\prime)$ are not only orthogonal, but are in fact orthogonal complements. 
\end{proof}

Global duality motivates the following definition.

\begin{definition} The ratio $$\mathcal{T}(E/E^\prime) = \frac{\big | \Sel_\phi(E/\Q) \big |}{\big |\Sel_{\hat \phi}(E^\prime/\Q)\big |}$$ is called the \textbf{Tamagawa ratio} of $E$. \end{definition}

What is important for our application is that the Tamagawa ratio can be computed using a local product formula.

\begin{theorem}[Cassels]\label{prodform2}
The Tamagawa ratio $\mathcal{T}(E/E^\prime)$ is given by $$\mathcal{T}(E/E^\prime) = \prod_{v\text{ of } \Q}\frac{\left | H^1_\phi(\Q_v, C)\right |}{2}.$$ 
\end{theorem}
\begin{proof}
This is a combination of Theorem 1.1 and equations (1.22) and (3.4) in \cite{Cassels1965}. Alternatively, this follows from combining Lemma \ref{localduality} with Theorem 2 in \cite{Washington1978}.
\end{proof}

\begin{remark}
The product in Theorem \ref{prodform2} converges because $\frac{\left | H^1_\phi(\Q_p, C)\right |}{2} = 1$ for primes $p$ different from $2$ where $E$ has good reduction. More generally, because $H^1(\Q_v, C) \simeq Q_v^\times/(\Q_v^\times)^2$, $\left | H^1_\phi(\Q_p, C)\right | \le 8$ for all places $v$ of $\Q$.
\end{remark}

This next Lemma gives an easy formula for computing $\left | H^1_\phi(\Q_p, C)\right |$ for $p \ne 2$.
\begin{lemma}\label{lem:fudge}
For $p$ different from $2$, $\left | H^1_\phi(\Q_p, C)\right | = \frac{c_p^\prime}{c_p}$, where $c_p$ and $c_p^\prime$ are the Tamagawa fudge factors at $p$ for $E$ and $E^\prime$ respectively.
\end{lemma}
\begin{proof}
This is a combination of Lemmas 4.2.(2) and 4.3 in \cite{DD2012}.
\end{proof}

\section{Local Conditions}
\label{sec: local conditions}

If $E$ is an elliptic curve with a single point of order two, then $E$ is given by a model of the form $y^2=x^3+Ax^2+Bx$, where the point $(0,0)$ has order two. If we insist that we don't have both $p^2 \mid A$ and $p^4 \mid B$ for any prime $p$, then $E$ has a unique model of this form, and this model will be minimal except possibly at $2$.

Given such a model, we can easily read off the reduction type of $E$ at any prime $p \ne 2$.

\begin{proposition} \label{prop:classification} Let $p$ be a prime different from $2$.
\begin{enumerate}[(i)]
\item $E$ has good reduction at $E$ if $p \nmid B(A^2 - 4B)$.
\item $E$ has additive reduction at $E$ if $p \mid B$ and $p\mid A^2 - 4B$.
\item $E$ has multiplicative reduction at $p$ if $p$ divides exactly one of $A^2 - 4B$ and $B$. If $p \mid A^2-4B$, then this reduction is split if and only if $\legendre{-2AB}{p}= 1$; if $p \mid B$, then this reduction is split if and only if $\legendre{B}{p} = 1$.
\end{enumerate}
\end{proposition}
\begin{proof}
This follows easily from Tate's algorithm. See Section IV.9 in \cite{ATAEC}, for example.
\end{proof}

Proposition \ref{prop:classification} tells us that for a given prime $p$, the probability that a curve $E$ has multiplicative reduction at $p$ is $\frac{2}{p}+O(\frac{1}{p^2})$ and the probability $E$ has additive reduction at $p$ is $O(\frac{1}{p^2})$. This leads us to expect that the dominant contribution towards $\T(E/E^\prime)$ will come from primes of multiplicative reduction and we therefore compute the contribution at such places.

\begin{proposition}
Suppose that $E$ has multiplicative reduction at $p$ different from $2$. Then $$|H^1_\phi(\Q_p,C)| = \frac{c_p^\prime}{c_p} = \left \{ \begin{array}{cl}
4 & \text{if  } \ord_p (A^2 - 4B) \text{ is odd or } \legendre{-2AB}{p}= 1  \\
1 & \text{if } \ord_p B \text{ is odd or } \legendre{B}{p} = 1 \\
2 & \text{otherwise} \end{array} \right . $$
\end{proposition}
\begin{proof}
It is easy to check that $E$ and $E^\prime$ have Kodaira types $I_n$ and $I_{n^\prime}$ respectively, where $n = \ord_p (A^2 - 4B) + 2\ord_p B$ and $n^\prime = 2\ord_p (A^2 - 4B) + \ord_p B$. The equality on the right is then immidiate from Tate's algorithm combined with Proposition \ref{prop:classification}.(iii). The equality on the left is Lemma \ref{lem:fudge}.
\end{proof}


\section{The distribution of the Tamagawa ratio}
\label{sec:distribution}

Recall from Theorem \ref{prodform2} that the Tamagawa ratio $\mathcal{T}(E/E^\prime)$ can be expressed as a product of local factors,
\[
\mathcal{T}(E/E^\prime) = \prod_{v \mid 2 \Delta \infty} \mathcal{T}_v(E/E^\prime), 
\]
one for each place of bad reduction.  For the elliptic curve $E_{A,B}:y^2=x^3+Ax^2+Bx$ with a two-torsion point, we can therefore express $t(A,B)=\mathrm{ord}_2 \mathcal{T}(E/E^\prime)$ as a sum over such places,
\[
t(A,B) = \sum_{v\mid 2 \Delta_{A,B} \infty} t_v(A,B),
\]
which we can further split as
\[
t(A,B) = t_{\mathrm{mult}}(A,B) + t_{\mathrm{add}}(A,B) + O(1),
\]
where $t_{\mathrm{mult}}(A,B)$ is the contribution from the primes of multiplicative reduction, $t_{\mathrm{add}}(A,B)$ is the contribution from the primes of additive reduction, and the $O(1)$ term comes from the places $2$ and $\infty$.  As observed earlier, Proposition \ref{prop:classification} shows that the probability that a given prime $p$ is of multiplicative reduction is $2/p+O(1/p^2)$ and the probability it is of additive reduction is $O(1/p^2)$.  (Though it is likely clear that these are roughly the correct probabilities, Lemma \ref{lem:density} below makes this precise.)  We therefore expect that the primes of additive reduction will have a finite contribution to the distribution of the Tamagawa ratio, owing to the convergence of $\sum 1/p^2$, whereas the primes of multiplicative reduction will not.  Before establishing this, we make our intuition on probabilities precise.

\begin{lemma}
\label{lem:density}
For each prime $p$ and for any integers $a$ and $b$, let 
\[
\delta(p;(a,b)) := \left\{ \begin{array}{ll} \displaystyle \frac{p^4}{p^6-1} & \text{if } p\nmid a \text{ or } p\nmid b, \text{ and} \\ \displaystyle \frac{p^4-1}{p^6-1} & \text{if } p\mid a \text{ and } p\mid b. \end{array} \right.
\]
Let $q$ be a squarefree integer, and let $\delta(q;(a,b)) = \prod_{p\mid q} \delta(p;(a,b))$.  We then have that
\[
\#\{(A,B)\in \mathcal{E}(X) : (A,B) \equiv (a,b) \pmod{q}\} = \delta(q;(a,b))\cdot \frac{4 X^{3/2}}{\zeta(6)} + O(q^2X+q^6X^{3/8}),
\]
where $\zeta(s)$ is the Riemann zeta function.
\end{lemma}

\begin{proof}
For each prime $p$, consider the class $(a,b) \pmod{p}$.  If $(a,b)\not\equiv (0,0)\pmod{p}$, then it lifts to $p^6$ classes $\pmod{p^4}$, each of which is occupied by elements of $\mathcal{E}(X)$.  On the other hand, if $(a,b)\equiv (0,0)\pmod{p}$, there will be $p^2$ lifts $\pmod{p^4}$ which are not occupied.  Thus, a class $(a,b)\pmod{q}$, with $q$ squarefree, can be lifted $\pmod{q^4}$ in exactly
\[
\prod_{\begin{subarray}{c} p\mid q \\ p\nmid a \text{ or } p\nmid b\end{subarray}} p^6 \prod_{\begin{subarray}{c} p\mid q \\ p\mid a \text{ and } p\mid b\end{subarray}} (p^6-p^2)
\]
ways that will occur in $\mathcal{E}(X)$.  Let $(a^\prime,b^\prime)$ be such a lift.  We then have that
\begin{eqnarray*}
\sum_{\begin{subarray}{c} (A,B)\in\mathcal{E}(X): \\ (A,B) \equiv (a^\prime,b^\prime) \pmod{q^4} \end{subarray}} 1
	&=& \sum_{\begin{subarray}{c} B^2\leq X \\ B \equiv b^\prime \pmod{q^4}\end{subarray}} \sum_{\begin{subarray}{c} |A| \leq X \\ A \equiv a^\prime \pmod{q^4} \\ p^2\nmid A \text{ if } p^4 \mid B \end{subarray}} 1 \\
	&=& \sum_{\begin{subarray}{c} B^2\leq X \\ B \equiv b^\prime \pmod{q^4}\end{subarray}} \left[\frac{2X}{q^4} \prod_{p^4\mid B, p\nmid q} \left(1-\frac{1}{p^2}\right) + O\left(\prod_{p^4 \mid B, p\nmid q} p^2\right)  \right] \\
	&=:& \frac{2X}{q^4} \sum_{\begin{subarray}{c} B^2\leq X \\ B \equiv b^\prime \pmod{q^4}\end{subarray}} f_q(B) + O\left(\sum_{\begin{subarray}{c} B^2\leq X \\ B\equiv b^\prime\pmod{q^4}\end{subarray}} \prod_{p^4\mid B, p\nmid q} p^2 \right),
\end{eqnarray*}
say, where $f_q(B)$ is multiplicative.  Let $g_q:=f_q * \mu$, so that $f_q(B)=\sum_{d\mid B} g_q(d)$; note that $g_q(d)=0$ if $(d,q)>1$.  The summation in the main term is thus
\begin{eqnarray*}
\sum_{\begin{subarray}{c} B^2\leq X \\ B \equiv b^\prime \pmod{q^4}\end{subarray}} f_q(B)
	&=& \sum_{\begin{subarray}{c} d \leq X^{1/2} \\ (d,q)=1 \end{subarray}} g_q(d) \sum_{\begin{subarray}{c} |B| \leq X^{1/2}/d \\ B \equiv b^\prime d^{-1} \pmod{q^4} \end{subarray}} 1 \\
	&=&\frac{2X^{1/2}}{q^4} \sum_{\begin{subarray}{c} d\leq X^{1/2} \\ (d,q)=1 \end{subarray}} \frac{g_q(d)}{d} + O\left(\sum_{d\leq X^{1/2}} |g_q(d)|\right).
\end{eqnarray*}
We note that the Dirichlet series $L(s,g_q)$ satisfies
\[
L(s,g_q) = \prod_{p\nmid q} \left(1-p^{-2-4s}\right) \text{ and } L(1,g_q) = \zeta(6)^{-1} \prod_{p\mid q} \left(1-p^{-6}\right)^{-1},
\]
so that 
\[
\sum_{\begin{subarray}{c} B^2\leq X \\ B\equiv b^\prime \pmod{q^4} \end{subarray}} f_q(B) = \frac{2X^{1/2}}{q^4 \zeta(6)} \prod_{p\mid q} \left(1-p^{-6}\right)^{-1} + O(1).
\]
Similarly, we also find that
\[
\sum_{\begin{subarray}{c} B^2\leq X \\ B\equiv b^\prime\pmod{q^4}\end{subarray}} \prod_{p^4\mid B, p\nmid q} p^2 \ll \frac{X^{1/2}}{q^4}+X^{3/8},
\]
whence
\[
\sum_{\begin{subarray}{c} (A,B)\in\mathcal{E}(X): \\ (A,B) \equiv (a^\prime,b^\prime) \pmod{q^4} \end{subarray}} 1 = \frac{4X^{3/2}}{q^8 \zeta(6)} \prod_{p\mid q}\left(1-p^{-6}\right)^{-1} + O\left(\frac{X}{q^4}+X^{3/8}\right).
\]
Summing over lifts $(a^\prime,b^\prime)$, the result follows.
\end{proof}

We are now ready to prove Theorem \ref{thm:tamagawa}.

\begin{proof}[Proof of Theorem \ref{thm:tamagawa}]
We proceed via the method of moments, adapting an approach due to Billingsley \cite{Billingsley1974} to prove the classical Erd\H{o}s-Kac theorem.  

We first note that the set of $(A,B)\in\mathcal{E}(X)$ for which $A^2-4B$ is a square is $O(X)$, and so, in view of the fact that $\#\mathcal{E}(X)\sim 4X^{3/2}/\zeta(6)$, such $(A,B)$ will have no contribution to the limiting distribution.  We therefore assume in the sequel that $A^2-4B$ is not a square, which amounts to assuming that $(0,0)$ is the only non-trivial two-torsion point on $E_{A,B}/\mathbb{Q}$.

Consider the functions
\[
g_1(A,B) := \sum_{p \mid A^2-4B} 1 \text{ and } g_2(A,B) := \sum_{p\mid B} 1,
\]
and note that 
\[
t(A,B) = g_1(A,B) - g_2(A,B) + t_{\mathrm{add}}(A,B)+O\left(\sum_{\begin{subarray}{c} p^2 \mid {A^2-4B} \\ \text{or } p^2 \mid B \end{subarray}} 1 \right), 
\]
where the implied constant may to be taken to be $1$.  Let $T=\epsilon \sqrt{\log\log X}$, and consider the error term.  There are $O(X^{3/2}/p^2)$ pairs $(A,B)\in\mathcal{E}(X)$ with either $p^2\mid A^2-4B$ or $p^2\mid B$, whence there are $O(X^{3/2}/T)$ pairs satisfying these divisibility conditions for some prime $p>T$.  For the remaining full-density subset, the contribution from the sum is manifestly $\leq T$.  Similarly, there are $O(X^{3/2}/T)$ pairs $(A,B)$ for which $t_\mathrm{add}(A,B) > T$.  We will now show that $g_1(A,B)$ and $g_2(A,B)$ are asymptotically independent and normally distributed, each with mean and variance $\log\log X$, from which Theorem \ref{thm:tamagawa} therefore follows.

For any prime $p$, let $\rho(p) = (p^5-1)/(p^6-1)$.  A simple calculation with Lemma \ref{lem:density} reveals that $\rho(p)$ is both the probability that $(A,B)\in\mathcal{E}(X)$ satisfies $p\mid A^2-4B$ and the probability that $p \mid B$.  We therefore expect that $g_1(A,B)$ and $g_2(A,B)$ should be normal with mean $\mu(X)$ and variance $\sigma^2(X)$ both given by
\[
\mu(X),\sigma^2(X)=\sum_{p<X}\rho(p) = \log\log X + O(1).
\]

Let $z=X^\delta$ for some $\delta>0$.  For each odd prime $p<z$, denote by $D_p$ and $D_p^\prime$ random variables which are 1 with probability $\rho(p)$ and 0 with probability $1-\rho(p)$, and are such that
\[
\mathrm{Prob}(D_p=1 \text{ and } D_p^\prime=1) = \frac{p^4-1}{p^6-1}.
\]
In view of Lemma \ref{lem:density}, we think of $D_p$ and $D_p^\prime$ as modeling the events $p\mid B$ and $p\mid A^2-4B$.  If we set
\[
D(z) := \sum_{p<z} D_p \text{ and } D^\prime(z) := \sum_{p<z} D_p^\prime,
\]
the multidimensional central limit theorem (with Lindeberg's criterion, say) implies that, as $z\to\infty$, $D(z)$ and $D^\prime(z)$ become independent and normally distributed with mean and variance each $\log\log z$.  We will show that the $(k_1,k_2)$-mixed moment of $g_1(A,B)$ and $g_2(A,B)$ agrees with the $(k_1,k_2)$-mixed moment of $D(z)$ and $D^\prime(z)$, and since mixed moments determine the multinormal distribution, the result will follow.

First, let $g_1(A,B;z)$ and $g_2(A,B;z)$ be defined by
\[
g_1(A,B;z) := \sum_{\begin{subarray}{c} p \mid A^2-4B \\ p<z \end{subarray}} 1 \text{ and } g_2(A,B;z) := \sum_{\begin{subarray}{c} p\mid B \\ p<z \end{subarray}} 1.
\]
For any integers $k_1,k_2\geq 0$, set $z=X^{1/7(k_1+k_2)}$.  Using Lemma \ref{lem:density}, we compute that
\begin{eqnarray*}
\frac{1}{\#\mathcal{E}(X)}\sum_{(A,B)\in\mathcal{E}(X)} g_1(A,B;z)^{k_1}g_2(A,B;z)^{k_2} 
	&=& \sum_{\begin{subarray}{c} p_1,\dots,p_{k_1}<z \\ q_1,\dots,q_{k_2}<z \\ \text{prime} \end{subarray}} \frac{1}{\#\mathcal{E}(X)}\sum_{\begin{subarray}{c} (A,B)\in\mathcal{E}(X) : \\ p_i \mid A^2-4B\, \forall i \\ q_j \mid B\, \forall j \end{subarray}} 1 \\
	&=& \sum_{\begin{subarray}{c} p_1,\dots,p_{k_1}<z \\ q_1,\dots,q_{k_2}<z \\ \text{prime} \end{subarray}} P(\mathbf{p};\mathbf{q}) + O(X^{-1/14})
\end{eqnarray*}
where $P(\mathbf{p};\mathbf{q})$ is the density of $(A,B)\in\mathcal{E}(X)$ for which each $p_i \mid A^2-4B$ and each $q_j \mid B$.  We also observe that
\[
\mathbb{E}\left[D(z)^{k_1}D^\prime(z)^{k_2}\right] = \sum_{\begin{subarray}{c} p_1,\dots,p_{k_1}<z \\ q_1,\dots,q_{k_2}<z \\ \text{prime} \end{subarray}} P(\mathbf{p};\mathbf{q})
\]
by the construction of $D_p,D_p^\prime$.  We therefore have, letting $\mu(z)=\log\log z$, that
\begin{align*}
\frac{1}{\#\mathcal{E}(X)} &\sum_{(A,B)\in\mathcal{E}(X)} \left(g_1(A,B;z)-\mu(z)\right)^{k_1}\left(g_2(A,B;z)-\mu(z)\right)^{k_2} \\
	&= \sum_{j_1=0}^{k_1} \sum_{j_2=0}^{k_2} (-\mu(z))^{j_1+j_2} \left({k_1}\atop{j_1}\right)\left({k_2}\atop{j_2}\right) \frac{1}{\#\mathcal{E}(X)}\!\! \sum_{(A,B)\in\mathcal{E}(X)}\!\! g_1(A,B;z)^{k_1-j_1}g_2(A,B;z)^{k_2-j_2} \\
	&= \sum_{j_1=0}^{k_1} \sum_{j_2=0}^{k_2} (-\mu(z))^{j_1+j_2} \left({k_1}\atop{j_1}\right)\left({k_2}\atop{j_2}\right) \mathbb{E}\left[D(z)^{k_1-j_1}D^\prime(z)^{k_2-j_2}\right] + O(X^{-1/14}) \\
	&= \mathbb{E}\left[(D(z)-\mu(z))^{k_1}(D^\prime(z)-\mu(z))^{k_2}\right] + O(X^{-1/14}).
\end{align*}
Thus, $g_1(A,B;z)$ and $g_2(A,B;z)$ have the same moments as $D(z)$ and $D^\prime(z)$.  Finally, for $i=1,2$, we see that
\[
g_i(A,B)-\mu(X) = g_i(A,B;z)-\mu(z) + O(1),
\]
so that, by the binomial theorem and the Cauchy-Schwarz inequality, 
\begin{align*}
\sum_{(A,B)\in\mathcal{E}(X)} &\left(g_1(A,B)-\mu(X)\right)^{k_1}\left(g_2(A,B)-\mu(X)\right)^{k_2} \\
	&= (1+O(\mu(X)^{-1/2}))\sum_{(A,B)\in\mathcal{E}(X)} \left(g_1(A,B;z)-\mu(z)\right)^{k_1}\left(g_2(A,B;z)-\mu(z)\right)^{k_2},
\end{align*}
Thus, the mixed moments of $g_1(A,B)$ and $g_2(A,B)$ converge to those of $D(z)$ and $D^\prime(z)$, and the result is proved.
\end{proof}

\bibliographystyle{alpha}
\bibliography{tamagawa}

\end{document}